\documentclass[reqno]{amsart}
\usepackage[noadjust]{cite}

\usepackage{amssymb}
\usepackage{amsmath}
\usepackage{amsfonts}
\usepackage{graphicx}
\usepackage{enumerate}
\usepackage{enumitem}
\usepackage{mathtools}
\usepackage[numbers,sort&compress]{natbib}
\usepackage{amsthm}
\usepackage{tikz}
\usepackage{theoremref}
\usepackage{float}
\usepackage[hyphens]{url} 
\usepackage{caption}
\date{}
\restylefloat{table}
\makeatletter
\def\thmheadbrackets#1#2#3{%
	\thmname{#1}\thmnumber{\@ifnotempty{#1}{ }\@upn{#2}}%
	\thmnote{ {\the\thm@notefont[#3]}}}
\makeatother

\newtheoremstyle{brakets}
{}
{}
{\itshape}
{}
{\bfseries}
{.}
{ }
{\thmheadbrackets{#1}{#2}{#3}}

\theoremstyle{brakets}

\newtheorem{thm}{Theorem}

\newtheorem{cor}[thm]{Corollary}

\def\({\left(}
\def\){\right)}

\newtheorem{lema}{Lemma}[section]

\newtheorem*{teorema*}{Theorem}

\newtheorem{remark}[lema]{Remark}

\newtheorem{example}[lema]{Example}

\newtheorem{theorem}[lema]{Theorem}

\newtheorem{definition}[lema]{Definition}

{\par \hfill \fbox{}}
{\par \hfill \fbox{}}
{\par \hfill }
{\par \hfill }

\numberwithin{equation}{section}

\begin{document}
	\title[Bicomplex numbers]{ Bicomplex Paley Weiner Theorem }
	\author{\bf{Sanjay Kumar}}
	\address{Department of Mathematics, Central University of Jammu, Rahya-Suchani (Bagla)-181 143, Jammu (J$\&$K), INDIA.}
	\email{sanjaykmath@gmail.com, sanjay.math@cujammu.ac.in}
	{\author{\bf{Stanzin Dolkar}}
	\address{Department of Mathematics, Central University of Jammu, Rahya-Suchani (Bagla)-181 143, Jammu (J$\&$K), INDIA.}
	\email{stanzin.math@cujammu.ac.in}
		
		\thanks{\textnormal{The work of the first author is supported by NBHM (DAE) Grant No. 2/11/41/2017/R$\&$ D-II/3480 and the second author is supported by UGC Grant. No. 1116/CSIR-UGC NET DEC.2018}}


		\begin{abstract}
			In this paper, we study the bicomplex version of the Paley-Weiner theorem and the Cauchy integral formula in the upper half-plane.
		\end{abstract}

		\maketitle

		\section{Introduction and Preliminaries}
		The study of bicomplex numbers started in 1892 when Segre \cite{segre1892} found that the property of commutativity had been missing from the skew field of quaternions. The Quaternions were first introduced by W. R. Hamilton in 1844. The study of bicomplex numbers has always been an active field of research. Segre was inspired by the works of Hamilton, and then he introduced a new number system called the bicomplex numbers. \\
		The work of J. D. Riley in \cite{riley1953} has further developed the theory of functions with bicomplex variables. Also, without forgetting to mention the work of G. B. Price  \cite{price1991} who provided us with a very powerful method to study holomorphic functions with bicomplex variables.  \\
		We denote the set of bicomplex numbers by $ \mathbb{BC} $ and define it as follows:
		$$ \mathbb{BC} = \{ z_{1} + j z_{2} ~:~ z_{1}, z_{2} \in \mathbb{C} \}, $$
		where $ \mathbb{C} $ is the set of complex numbers. Therefore, bicomplex numbers are sometimes called complex numbers with complex coefficients. The set of complex numbers has the imaginary unit $ i.$ As we see that $ \mathbb{BC} $ has $i, j$ as its imaginary units, and these two imaginary units are commuting, i.e., $ij=ji$ and also $i^{2} =j^{2} =-1.$ The bicomplex numbers can be added and multiplied. Thus, both addition and multiplication are commutative and associative. \\
		Another important set of numbers is the set of hyperbolic numbers, which can be defined independently of $ \mathbb{BC}.$ We denote the set of hyperbolic numbers by
		$$ \mathbb{D} = \{ a + k b : a,b\in \mathbb{R} \}, $$
		where $k$ is called the hyperbolic unit with $ k^{2} = 1 .$ The set of hyperbolic numbers is also called split-real numbers, Lorentz numbers, perplex numbers, etc. These were first introduced by Cockle \cite{cockle1849}. While working with $ \mathbb{BC}, $ we encounter $ ij=k .$ Thus, we realize that there exists a subset of the bicomplex numbers such that it is isomorphic to the set of split-real numbers. Thus, we can define the set of split-real numbers as
		$$ \mathbb{D} = \{ a+ ijb~:~ a,b \in \mathbb{R} \}. $$\\
		Another important feature of the bicomplex numbers $ \mathbb{BC} $ is the presence of idempotent units $ e$ and $e^{\dagger},$ which makes it possible to represent the bicomplex numbers in their idempotent form. The idempotent units are also called special zero divisors as
		$$ e= \frac{1+ij}{2}\textnormal\;\;{and} \;\; e^{\dagger} = \frac{1-ij}{2}$$
		and
		$$e.e^{\dagger}=0.$$
		Thus, each of them is a zero divisor. Also,
		$$ e^{2}=1 \textnormal\;\;{and} \;\;(e^{\dagger})^{2} =1. $$
		Thus any bicomplex number $ Z$ can be represented as $ Z = e \beta_{1} + e^{\dagger} \beta_{2}, $ where $ \beta_{1}$ and $ \beta_{2} $ are complex numbers. This is called representing a bicomplex number in terms of its idempotent units.
		 
		Next, we describe a few representations of $ Z \in \mathbb{BC} .$ Any $Z \in \mathbb{BC} $ can be written as
		\begin{align} Z =& z_{1} + j z_{2}\label{a}\\
			=& \xi_{1} + k \xi_{2}\label{b}\\
			=& e \beta_{1} + e^{\dagger} \beta_{2}\label{c}\\
			=& a_{1} + i a_{2} + j a_{3} + k a_{4}.\label{d}
		\end{align}
		The equation (\ref{a}) determines $ Z$ as an element of $ \mathbb{C}^{2}(i) ,$ while the equation (\ref{d}) identifies $Z$ as an element of $ \mathbb{R}^{4},$ equation (\ref{c}) is the idempotent representation of $Z.$ Lastly, (\ref{b}) identifies $ Z$ with elements in $ \mathbb{D}^{2}= \mathbb{D} \times \mathbb{D},$ see\cite[Page 7]{alpay2014} for more details.
		We begin with the definition of the upper half plane in $ \mathbb{BC}.$
		
		\begin{definition}\cite{kumar2015}
			We denote the upper half plane in $ \mathbb{BC} $ by $\substack \prod_{\mathbb{BC}^{+}} $ and is defined as
			$$ \displaystyle \substack \prod_{\mathbb{BC}^{+}} = \{ Z \in \mathbb{BC} : Z = z_{1} + j z_{2}~ or~ Z = e \beta_{1} + e^{\dagger} \beta_{2} : ( \beta_{1} , \beta_{2} ) \in \displaystyle \substack \prod^{+} \times \displaystyle \substack \prod^{+} \}, $$
			where $ \displaystyle \substack \prod^{+}  = \{ z \in \mathbb{C}(i) : z = x + i y ~ and~ y > 0 \in \mathbb{C}(i) \}$ is the upper half plane in $ \mathbb{C}.$	
		\end{definition}
		\begin{definition}\cite{reyes2019}
			A set $G$ in $ \mathbb{BC} $ is called a product type set if $ G = e G_{1} + e^{\dagger} G_{2} ,$ where $ G_{1} = \Pi_{1,i}(G) $ and $ G_{2} = \Pi_{2,i}(G) ,$ where $ \Pi_{1,i}(G)$ and $ \Pi_{2,i}(G) $ are the idempotent projections of $G$ on $G_{1}$ and $ G_{2},$ respectively.
		\end{definition}
		
		\begin{definition}\cite{reyes2019}
			A bicomplex function $ F: G \subset \mathbb{BC} \longrightarrow \mathbb{BC} $ is said to be a product type if $ F(Z)= e f_{1}(\beta_{1}) + e^{\dagger} f_{2}(\beta_{2}), $ where each $ f_{k}^{'}s: G_{k^{'}s} \longrightarrow \mathbb{C} ;k=1,2,$ are complex valued functions.
		\end{definition}
		\begin{definition}\cite{kumar2017}
			Let $\mathfrak{M} $ be a $ \sigma$-algebra in a set $G. $ A hyperbolic real-valued bicomplex function $ m = e m_{1} + e^{\dagger} m_{2} $ defined on $ G $ is called a hyperbolic measure if $ m_{1} $ and $ m_{2} $ are real measures on $ \mathfrak{M} .$
		\end{definition}
		
		\begin{definition}
			Let $F$ be a bicomplex product type function defined on $ ( -\infty, \infty )_{\mathbb{D}}  .$ Then
			\begin{equation}\label{A}
				F( e \beta_{1} + e^{\dagger} \beta_{2}) = e F_{1}(\beta_{1}) + e^{\dagger} F_{2}(\beta_{2}),
			\end{equation}
			where $ (-\infty, \infty)_{\mathbb{D}}$ is a bicomplex domain of product type such that
			$$ (-\infty ,\infty )_{\mathbb{D}} = e (-\infty ,\infty ) + e^{\dagger} (-\infty ,\infty ).$$
			It is worth noting that any $ Z$ in the upper half plane $ \substack \prod_{\mathbb{BC}^{+}} $ can also be written as
			$$ Z = e \beta_{1} + e^{\dagger} \beta_{2} \in  \substack \prod_{\mathbb{BC}^{+}}~~\textnormal{if and only if}~~ \beta_{1} = z_{1}-iz_{2} \in  \substack \prod^{+} ~and~ \beta_{2} = z_{1} + i z_{2} \in  \substack \prod^{+}.$$
			Then a simple elaboration shows that
			$$\beta_{1}=z_{1}-iz_{2} = (x_{0}+ix_{1}) - i ( x_{2}+ i x_{3}) =( x_{0} + x_{3}) + i ( x_{1}-x_{2}). $$
			Thus,
			\begin{equation}\label{B}
				\beta_{1} \in  \substack \prod^{+} ~~\textnormal{if and only if}~~ x_{1} - x_{2} > 0
			\end{equation}
			and
			$$\beta_{2}=z_{1}+iz_{2} = (x_{0}+ix_{1}) + i ( x_{2}+ i x_{3}) =( x_{0} - x_{3}) + i ( x_{1}+x_{2}). $$
			Therefore,
			\begin{equation}\label{B1}
				\beta_{2} \in  \substack \prod^{+} ~~\textnormal{if and only if}~~ x_{1}+ x_{2} > 0
			\end{equation}
			\textnormal{Hence equation (\ref{B}) and (\ref{B1}) implies that} $ \beta_{1},\beta_{2} \in \substack \prod^{+}~~\textnormal{if and only if}~~ x_{1}> | x_{2}|.$
		\end{definition}
		
		Next, we define the $ \mathbb{D}-$integral of $F$ on $ ( - \infty, \infty )_{\mathbb{D}} \subset \substack \prod_{\mathbb{BC}^{+}} $ by
		$$ \int_{(-\infty,\infty)_{\mathbb{D}}} F(Z) dZ \odot dZ^{\dagger} = e \int_{(-\infty,\infty)} F_{1}(\beta_{1}) d\beta_{1} + e^{\dagger} \int_{(-\infty,\infty)} F_{2}(\beta_{2}) d\beta_{2} ~~; Z \in \substack \prod_{\mathbb{BC}^{+}}.$$
		
		Using this definition of $ \mathbb{D}-$integral, we say that a bicomplex function $ F $ on $ \substack \prod_{\mathbb{BC}^{+}} $ is $ \mathbb{D}-$square integrable if
		$$ \int_{(-\infty,\infty)_{\mathbb{D}}} \| F \|_{k}^{2} dm \odot dm^{\dagger} < \infty ,$$
		where $ dm $ is the four-dimensional Lebesgue measure such that $ dm = e dm_{1} + e^{\dagger} dm_{2}.$\\
		Using equation (\ref{A}), we can say that $ F $ is $ \mathbb{D}$-square integrable if and only if $ F_{1} $ and $ F_{2} $ are square integrable. That is,
		$$\int_{\prod^{+}}  |F_{i}|^{2} dm_{i} < \infty.$$

		We denote the space of all $ \mathbb{D} -$square integrable functions on $ \substack 
		\prod_{\mathbb{BC}^{+}} $ by $ L_{k}^{2} \big( ( - \infty,\infty )_{\mathbb{D}},dm \big) $ and consequently,
		\begin{equation}\label{B1*}
			L_{k}^{2} \big( ( -\infty, \infty )_{\mathbb{D}} ,dm \big) = e L^{2} ( ( -\infty, \infty ) , dm_{1} ) + e^{\dagger} L^{2} ( ( -\infty, \infty ), dm_{2}).
		\end{equation}
		The hyperbolic norm of $ F \in   L_{k}^{2} \big( ( -\infty,\infty)_{\mathbb{D}}, dm \big) $ is defined as
		$$
		\|F\|_{k,2}^{2} =  e \| F_{1} \|_{1,2}^{2} + e^{\dagger} \| F_{2} \|_{2,2}^{2}. $$
		That is, 
		$$\int_{(-\infty,\infty)_{\mathbb{D}}} \|F\|_{k}^{2} dm =  e \int_{-\infty}^{\infty} | f_{1} |^{2} dm_{1} + e^{\dagger} \int_{-\infty}^{\infty} |f_{2} |^{2} dm_{2}. 
		$$
		Also, bicomplex Bergman spaces on the bounded domain were introduced in \cite{perez2018}.
		
		\begin{theorem}\label{B2}
			For $ 1 \leqslant p \leqslant \infty,$ the Cauchy sequence $ \{F_{n}\}$ in $ L_{k}^{p}(dm)$ with limit $ F$ has a pointwise convergent subsequence, almost everywhere to $ F(x_{0},x_{3}).$
		\end{theorem}
		\begin{proof}
			The proof of the above theorem is quite simple. From the equation (\ref{B1*}), we have
			\begin{equation}\label{B3}
				L_{k}^{p}(dm)= e L^{p}(dm_{1}) + e^{\dagger} L^{p} ( dm_{2}).
			\end{equation}
			Let $ Z = e \beta_{1} + e^{\dagger} \beta_{2}.$
			Then, having known the fact that for every Cauchy sequences $ \{ F_{n,1}\}$ and $ \{F_{n,2}\} $ in $ L^{p}(dm_{1}) $ and $ L^{p}(dm_{2})$ with limits $ F_{1} $ and $ F_{2} $, has a convergent subsequences converging to $ F_{1}(Re(\beta_{1}))$ and $ F_{2}(Re(\beta_{2}),$ respectively. Thus, the theorem holds for every Cauchy sequences $\{F_{n}\}$ in $ L_{k}^{p}(dm).$
		\end{proof}

		\begin{cor}\label{PW14a}
			Let $\hat{F}$ be the bicomplex Fourier transform of the bicomplex function F. If $F$ lies in $ L_{k}^{2} $ and $ \hat{F} \in L_{k}^{1},$ then
			\begin{equation*}
				F(x_{0},x_{3}) = \int_{(-\infty,\infty)_{\mathbb{D}}} \hat{F}(t) \exp\{i(x_{0}+kx_{3})\} dm(t)~~~~a.e.
			\end{equation*}
		\end{cor}
		For recent work on bicomplex analysis and its applications, one can refer to \cite{alpay2014,bie2012,colombo13,elizarraras2015,Lu-Sh-Va12} and the references therein.
		
		\section{Bicomplex Fourier Transforms}
		
		The bicomplex Fourier transform for functions of bicomplex variables is studied in \cite{banerjee2015,bie2012,ghanmi2019}.
		The standard bicomplex Fourier transform is defined as
		$$ \mathcal{F}_{\mathbb{BC}}(F)(Z) = \frac{1}{\sqrt{2\pi}} \int_{(-\infty,\infty)_{\mathbb{D}}} \exp\{-itZ\} F(t) dt \odot dt^{\dagger},$$
		where $ Z = e \beta_{1} + e^{\dagger} \beta_{2}.$
		
	Now, using the idempotent units $ e$ and $ e^{\dagger},$ we have
	\begin{align*}
		\mathcal{F}_{\mathbb{BC}}(F(Z)) =& \frac{1}{2\pi} \int_{(-\infty,\infty)_{\mathbb{D}}} \exp\{-it(e\beta_{1}+e^{\dagger}\beta_{2})\} dt \odot dt^{\dagger}\nonumber\\
		=& \frac{1}{2\pi} e \int_{-\infty}^{\infty} \exp\{-it\beta_{1}\} F_{1} dt + \frac{1}{2\pi} e^{\dagger} \int_{-\infty}^{\infty} \exp\{ -it\beta_{2}\} F_{2}(t) dt^{\dagger}\nonumber\\
		=& e \mathcal{F_{1}}(F_{1}) + e^{\dagger} \mathcal{F_{2}}(F_{2}).\nonumber\\
	\end{align*}

	\begin{example}\cite{banerjee2015}
		Consider $ F(t) = \exp\{-\|t\|_{k}\}.$ 
		Then,\\
		$$ \mathcal{F}_{\mathbb{BC}} ( F(t) ) = \frac {2} { 1 + Z^{2}} ~~;~~Z = z_{1} + j z_{2}, $$
		where $\mathcal{F}_{\mathbb{BC}}(F(t)) = \hat{F}(Z) = e \hat{F}_{1}(z_{1}) + e^{\dagger} \hat{F}_{2} (z_{2}) , \quad \hat{F}_{1}(z_{1}) = \frac{2}{1+z_{1}^{2}} $\quad  and \quad $ \hat{F}_{2} (z_{2}) = \frac{2}{1+ z_{2}^{2}} $ such that $ \hat{F}_{1}(z_{1}) $ and $ \hat{F}_{2}(z_{2}) $ are holomorphic in $ -1 < Img (z_{1}) $ and $ Img(z_{2}) < 1 .$
	\end{example}
	The above example shows that often, $ \hat{F} $ can be extended to a function holomorphic in some regions of $ \mathbb{BC} .$ Next, keeping in mind that $ \exp\{itZ\} $ is a holomorphic function of $ Z,$ we can expect and at the same time discuss a few conditions on $ F,$ when imposed on  $F$ turn it's bicomplex Fourier transform $ \hat{F}(t) $ into a holomorphic function in certain regions of $ \mathbb{BC} .$\\
	For the above claim, let $ \mathfrak{F} \in L_{k}^{2} ( ( -\infty,\infty)_{\mathbb{D}}, dm)$ such that $ \mathfrak{F}(t) = 0 $ on $ ( -\infty, 0)_{\mathbb{D}} .$ Then define
	\begin{equation}\label{p4}
		F(Z) = \int_{(0, \infty)_{\mathbb{D}}} \mathfrak{F}(t) \exp\{ itZ\} dt \odot dt^{\dagger},
	\end{equation}
	where $Z$ lies in the bicomplex upper half plane $ \substack \prod_{\mathbb{BC}^{+}}.$ Then
	\begin{align*}
		\exp\{itZ\} =& \exp \{ it(e\beta_{1} + e^{\dagger} \beta_{2})\}\nonumber\\
		=& e \exp \{ it\beta_{1}\} + e^{\dagger} exp \{ it \beta_{2} \} \nonumber\\
		=& e \exp \{ it [ ( x_{0} + x_{3}) + i ( x_{1}- x_{2} ) ]\} + e^{\dagger}\exp \{ it [ ( x_{0}-x_{3}) + i ( x_{1} + x_{2} )]\}.\nonumber\\
	\end{align*}
	Therefore, if $ Z \in  \substack \prod_{\mathbb{BC}^{+}}, $ then
	\begin{align*}
		\|\exp\{itZ\}\|_{k} =&\|  e \exp \{ it [ ( x_{0} + x_{3}) + i ( x_{1}- x_{2} ) ]\} \\ + & e^{\dagger}\exp \{ it [ ( x_{0}-x_{3}) + i ( x_{1} + x_{2} )] \|_{k}  \nonumber\\
		\leqslant & e \| \exp \{ it(x_{0} + x_{3}) \} \exp \{-t(x_{1} - x_{2}) \}\|_{1} \\ & e^{\dagger} \| \exp \{ it ( x_{0} - x_{3} ) \} \exp \{ -t(x_{1} + x_{2}) \} \|_{2} \nonumber\\
		\leqslant & e \|\exp \{ -t(x_{1}-x_{2})\}\|_{1} + e^{\dagger} \| \exp \{ -t ( x_{1} + x_{2} )\} \|_{2} \nonumber\\
		=&  e \| \exp \{ - t Img(\beta_{1}) \} \|_{1} + e^{\dagger} \| \exp \{ - t Img(\beta_{2} ) \} \|_{2} \nonumber\\\
		=& \exp \{ (e(-t Img\beta_{1})) + e^{\dagger} (-t Img \beta-{2} ) \} \nonumber\\
		=& \exp \{ - t ( e ( x_{1} - x_{2}) + e^{\dagger} ( x_{1} + x_{2}) ) \} \nonumber\\
		=& \exp \{ - t (x_{1} - k x_{2}) \}. \nonumber
	\end{align*}
	Hence, (\ref{p4}) exists and is well-defined. \\
	
	From equation (\ref{p4}),
	\begin{align}
		F(Z) =& \int_{(0,\infty)_{\mathbb{D}}}\mathfrak{F}(t) \exp\{ itZ\} dt \odot dt^{\dagger}\nonumber\\
		=& e \int_{0}^{\infty} \mathfrak{F}_{1} (t) \exp\{ it \beta_{1} \} dt + e^{\dagger} \int_{0}^{\infty} \mathfrak{F}_{2} (t) \exp\{it\beta_{2}\} dt^{\dagger}\nonumber\\
		=& e F_{1}(\beta_{1}) + e^{\dagger} F_{2} (\beta_{2}).\label{p4A} 
	\end{align}
	
	Thus $F$ is holomorphic on $ \substack \prod_{\mathbb{BC}^{+}} ,$ as each $F_{i}$ is holomorphic in $ \substack \prod^{+}.$ Here each $ F_{i}$ is defined as
	$$ F_{i}(\beta_{i}) = \int_{0}^{\infty} \mathfrak{F}_{i}(t) \exp\{ it \beta_{1}\} .$$ For more details, see \cite{rudin1966}.\\
	Next, we show that the restrictions of these functions to the horizontal lines in $ \substack \prod_{\mathbb{BC}^{+}} $ is bounded in $ L_{k}^{2}\big( ( -\infty, \infty)_{\mathbb{D}}, dm \big).$ Let $ Z \in \substack \prod_{\mathbb{BC}^{+}}.$ Then $ Z = e \beta_{1} + e^{\dagger} \beta_{2} = x_{0} +ix_{1} + jx_{2} + kx_{3},$ and from equation (\ref{p4A}), we have
	\begin{equation}\label{C}
		F(Z) = e F_{1}(\beta_{1}) + e^{\dagger} F_{2}(\beta_{2}),
	\end{equation}
	where each $ F_{1} $ and $ F_{2} $ are of the form
	\begin{align}\label{D}
		F_{1}(\beta_{1}) =& \int_{0}^{\infty} \mathfrak{F}_{1} (t) \exp\{ -t(x_{1}-x_{2})\} \exp\{ it(x_{0}+x_{3})\} dt\\
		F_{2}(\beta_{2}) =& \int_{0}^{\infty} \mathfrak {F}_{2} (t) \exp\{ -t(x_{1}+x_{2})\} \exp\{ it(x_{0}-x_{3})\} dt.
	\end{align}
	Then $F_{1} $ and $F_{2}$ are the restrictions to the horizontal lines in $ \substack \prod^{+}$ and from \cite{rudin1966}, we see that these restrictions form a bounded set in $ L^{2}((-\infty,\infty),dm_{i})~;i=1,2.$ Hence from equation (\ref{C}), we see that the restrictions of $F$ to the horizontal lines in $\substack  \prod_{\mathbb{BC}^{+}} $ form a bounded set in $ L_{k}^{2}((-\infty,\infty)_{\mathbb{D}},dm).$ Thus, the following remark concludes that:

	\begin{remark}\label{p9}
		The restrictions $F_{1} $ and $F_{2}$ of $F$ to the horizontal lines also form a bounded set in $ L_{k}^{2}\big( ( -\infty, \infty)_{\mathbb{D}}, dm \big).$ For more details, we refer to  \cite{rudin1966}. 
	\end{remark}
	\section{Bicomplex Paley-Weiner Theorem}
	In this section, we generalize the Paley-Weiner theorem in a bicomplex setting. The basis of the Paley-Weiner theorem lies in the outstanding fact that the converse of the Remark \ref{p9} is also true.
	\begin{theorem}
		Let $ F : \substack \prod_{\mathbb{BC}^{+}} \longrightarrow \mathbb{BC} $ be a holomorphic function on $ \substack \prod_{\mathbb{BC}^{+}}$ and
		\begin{equation*}
			\sup_{\substack {\mathbb{D}\\ x_{1} > |x_{2} |}}  \frac{1}{2\pi} \int_{(-\infty,\infty)_{\mathbb{D}}} \| F(Z)\|_{k}^{2} d x_{0} = M < \infty .
		\end{equation*}
		Then there exists $ \mathfrak{F} \in L_{k}^{2} \big( ( -\infty, \infty)_{\mathbb{D}}, dm \big) $ such that
		$$ F(Z) = \int_{(0,\infty)_{\mathbb{D}}} \mathfrak{F}(t) \exp\{ itZ \}dt \odot dt^{\dagger}, $$
		where $ Z $ lies in $ \substack \prod_{\mathbb{BC}^{+}} $ with $ Z = e \beta_{1} + e^{\dagger} \beta_{2} = x_{0} + ix_{1}+jx_{2} + k x_{3}$ and
		$$ \int_{(0,\infty)_{\mathbb{D}}} \| F(t) \|_{k}^{2} dt = M $$ for some constant $M.$
	\end{theorem}
	\begin{proof}
		We begin the proof with the supposition that such a holomorphic $ L_{k}^{2} $ function exists, say $ \mathfrak{F} $
		and let $F$ be a bicomplex holomorphic function defined on the upper half plane $ \substack \prod_{\mathbb{BC}^{+}} .$ Then
		$$ F(Z) = e F_{1}(\beta_{1}) + e^{\dagger} F_{2}(\beta_{2}),$$
		where $F_{l}$ for $l=1,2$ is holomorphic on the complex upper half-plane $\substack  \prod^{+}.$
		Then, by the Paley-Weiner Theorem, for each $F_{1}$ and $ F_{2}  \in H(\substack \prod^{+}),$ there exist $ \mathfrak{F}_{1} $ and $ \mathfrak{F}_{2} $ in $ L^{2}(0,\infty) $ such that each $ F_{1}((x_{0} + x_{3} )+ i ( x_{1}- x_{2}) ) $ and $ F_{2}(( x_{0} - x_{3}) + i ( x_{1} + x_{2})) $ are the inverse fourier transform of  $ \mathfrak{F}_{1} \exp\{ - ( x_{1}-x_{2}) t \} $ and $ \mathfrak{F}_{2} \exp\{ -(x_{1}+x_{2} ) t \},$ respectively, that is,
		
		\begin{equation}\label{PW1}
			F_{1}(\beta_{1}) = \mathcal{F}^{-1} \big( \mathfrak{F_{1}}(t) \exp\{-(x_{1}-x_{2})\}\big)
		\end{equation}
		and
		\begin{equation}\label{PW2}
			F_{2}(\beta_{2}) = \mathcal{F}^{-1} \big ( \mathfrak{F_{2}} (t) \exp\{-(x_{1}+x_{2}\} \big).
		\end{equation}
		Then, by the inversion formula, we have
		\begin{equation}\label{PW3}
			\mathfrak{F_{1}}(t) = \mathcal{F}\{F_{1}(\beta_{1})\exp\{ - ( x_{1}-x_{2})t\}\}
		\end{equation}
		\begin{equation}\label{PW3A}
			\mathfrak{F}_{2}(t) = \mathcal{F}\{ F_{2}(\beta_{2}) \exp\{ - ( x_{1} + x_{2})t\}\}.
		\end{equation}
		Now, from equations (\ref{PW3}) and (\ref{PW3A}), we get
		\begin{align*}
			e \mathfrak{F}_{1} (t) + e^{\dagger} \mathfrak{F}_{2}(t) =& e \bigg( \frac{1}{2\pi} \int_{-\infty}^{\infty} F_{1}(\beta_{1}) \exp\{ -it(x_{0}+x_{3})- t x_{1} + t x_{2} \} dx_{0}\bigg) \\&+ e^{\dagger}\bigg( \frac{1}{2\pi} \int_{-\infty}^{\infty} F_{2}(\beta_{2}) \exp\{ -it(x_{0}-x_{3})- t x_{1} + t x_{2} \} dx_{0}^{\dagger}\bigg)\nonumber\\
			=& e \frac{1}{2\pi} \int_{-\infty}^{\infty} F_{1}(\beta_{1}) \exp\{ -it\beta_{1}\}d
			x_{0} + e^{\dagger} \frac{1}{2 \pi} \int_{-\infty}^{\infty} F_{2}(\beta_{2}) \exp\{ -it\beta_{2}\}dx_{0}^{\dagger} \nonumber\\
			=& \frac{1}{2\pi} \int_{(-\infty,\infty)_{\mathbb{D}}} F(Z) \exp\{ -itZ \} d Z\odot dZ^{\dagger}\nonumber\\
			=& \mathfrak{F}(Z).
		\end{align*}
		Thus, for a bicomplex holomorphic function in $\substack  \prod_{\mathbb{BC}^{+}},$ we assumed the existence of an $ L_{k}^{2} $ function $ \mathfrak{F} $ such that
		\begin{equation}\label{PW4}
			\mathfrak{F}(Z) = \frac{1}{2\pi} \int F(Z) \exp\{-itZ\} dZ \odot dZ^{\dagger}.
		\end{equation}
		The integral in (\ref{PW4}) is the result of choosing a horizontal line in $ \substack \prod_{\mathbb{BC}^{+}} ,$ as the equations (\ref{PW1}) and (\ref{PW2}) are representations along the horizontal lines in $ \substack \prod^{+} .$ Now, we need to show that $ \mathfrak{F} \in L_{k}^{2}\big(( 0,\infty)_{\mathbb{D}}\big)$ is uniquely defined. So, we use the Cauchy theorem here.\\
		For this, let $\curlywedge_{\alpha}$ be a rectangular path in $ \substack \prod_{\mathbb{BC}^{+}} .$ Then $ \curlywedge_{\alpha} $ being a closed path, can be written as
		\begin{equation}\label{PW5}
			\curlywedge_{\alpha} = e \curlywedge_{\alpha_{1}} + e^{\dagger} \curlywedge_{\alpha_{2}},
		\end{equation}
		where $ \curlywedge_{\alpha_{1}} $ and $ \curlywedge_{\alpha_{2}} $ are rectangular paths in $ e \substack \prod^{+} $ and $ e^{\dagger} \substack \prod^{+},$ respectively. Using the equation (\ref{PW5}), we can assume the vertices of $ \curlywedge_{\alpha_{1}}$ as $ e ( \pm \alpha + i ), e^{\dagger}( \pm \alpha + i ) $ and $ e ( \pm \alpha + i y ), e^{\dagger} ( \pm \alpha + i y ),$ let
		$$ I = \int_{\curlywedge_{\alpha}} F(Z) \exp\{-itZ\} dZ \odot dZ^{\dagger}. $$
		Then,
		\begin{equation}\label{PW6}
			I = e \int_{\curlywedge_{\alpha_{1}}} F_{1} \exp\{-it\beta_{1}\} d\beta_{1} + e^{\dagger} \int_{\curlywedge_{\alpha_{2}}} F_{2}(\beta_{2}) \exp\{-it\beta_{2}\} d\beta_{2},
		\end{equation}
		where $ F = e F_{1}  + e^{\dagger} F_{2} $ such that $ F_{1},F_{2} \in H( \substack\prod^{+})$ and $ Z = e \beta_{1} + e^{\dagger} \beta_{2} $ such that $( \beta_{1},\beta_{2}) \in \substack \prod^{+} \times \substack \prod^{+} .$
		So, by using Cauchy's theorem, we get
		\begin{equation}\label{PW7}
			I = 0.
		\end{equation}
		Using the equation (\ref{PW6}), we have $ I = e I_{1} + e^{\dagger} I_{2} .$ Solving $I_{1} $ for the straight lines $ e( \gamma+i )$ to $ e( \gamma + iy ),$ we get a sequence $ ( \alpha_{1,j})_{j=1}^{\infty} $ such that $ I_{1}(\alpha_{1,j}) \longrightarrow 0 $ and $ I_{1}(-\alpha_{1,j}) \longrightarrow 0 $ as $ j \rightarrow \infty $ in $ e \substack \prod^{+} .$\\
		Similarly, for $ I_{2},$ we find a sequence $ (\alpha_{2,j})_{j=1}^{\infty} $ such that $ I_{2}( \alpha_{2,j})  \longrightarrow 0 $ as $ j \rightarrow \infty $ and $ I_{2}(-\alpha_{2,j}) \longrightarrow 0$ as $ j \rightarrow \infty $ in $ e^{\dagger} \substack \prod^{+} .$ Thus there must be a sequence $ \{ \alpha_{k,j}\}_{j=1}^{\infty} $ such that $ \{\alpha_{k,j}\}_{j=1}^{\infty} = e \{ \alpha_{1,j}\} _{j=1}^{\infty} + e^{\dagger} \{ \alpha_{2,j}\}_{j=1}^{\infty} $ and
		\begin{equation}\label{PW8}
			I(\alpha_{k,j}) \longrightarrow 0~~ \textnormal{and}~~ I(-\alpha_{k,j})\longrightarrow 0.
		\end{equation}
		Proceeding further, define
		$$ G_{j}( x_1 , x_2 , t) = \frac{1}{2\pi} \int_{-\alpha_{k,j}}^{\alpha_{k,j}} F(z_{1} + j z_{2}) \exp\{-it(x_{0} + k x_{3}) \} dx_{0}.$$
		
		Then, by equations (\ref{PW7}) and (\ref{PW8}), we get
		\begin{equation}\label{PW9}
			\displaystyle \lim_{\substack {{j\rightarrow \infty}}} \{ \exp\{-ktx_{2} + tx_{1} \} G_{j} ( x_{1},x_{2},t) - \exp\{ -kt+t\} G_{j} ( 1,1,t) \} = 0.
		\end{equation}
		
		Now, let $\mathcal{F}_{\mathbb{BC}} $ be the bicomplex fourier transform and writing $ F_{x_{1},x_{2}}(x_{0},x_{3}) $ for $ F( x_{0} + ix_{1} + jx_{2} + k x_{3} ) .$ Then $ F_{x_{1},x_{2}} $ lies in $ L_{k}^{2}(-\infty, \infty)_{\mathbb{D}}.$\\
		By the bicomplex Plancheral theorem \cite{ghanmi2019}, we have
		$$ \displaystyle \lim_{\substack{{j\rightarrow \infty}}} \int_{(-\infty, \infty)_{\mathbb{D}}} \| \hat{F}_{x_{1},x_{2}} (t) - G_{j}(x_{1},x_{2},t) \|_{k}^{2} dt \odot dt^{\dagger} = 0 .$$
		Thus, by Theorem \ref{B2}, the sequence $\{G_{j}(x_{1},x_{2},t)\}$ has a pointwise convergent subsequence that converges to $\hat{F}_{x_{1},x_{2}} (t)$ for almost every t.
		Now, defining
		\begin{equation}\label{PW9a}
			\mathfrak{F}(t) = \exp\{ -kt+t\} \hat{F}_{1,1} (t).  
		\end{equation} 
		From equation (\ref{PW9}), we have
		\begin{equation}\label{PW10}
			\mathfrak{F}(t) = \exp\{ -ktx_{2} + t x_{1} \} \hat{F}_{x_{1},x_{2}}(t).
		\end{equation}
		Thus again, from the Plancheral Theorem for $ \mathbb{BC},$ we have for every $ x_{1},x_{2} \in ( 0, \infty)_{\mathbb{D}},$
		\begin{align}\label{PW11}
			\int_{(-\infty, \infty)_{\mathbb{D}}} \exp\{ -2(-ktx_{2} + tx_{1} )\} \| \mathfrak{F}(t)\|_{k}^{2} dt \odot dt^{\dagger} = & \int_{(-\infty, \infty)_{\mathbb{D}}} \| \hat{F}_{x_{1},x_{2}}(t)\|_{k}^{2} dt \odot dt^{\dagger}\nonumber\\
			=& \frac{1}{2\pi} \int_{(-\infty,\infty)_{\mathbb{D}}} \|F_{x_{1},x_{2}}(x_{0},x_{3})\|_{k}^{2} dx_{0} \nonumber\\
			\leqslant&
			M.
		\end{align}
		If we let $ x_{2}, x_{1}  \rightarrow \infty ,$ then equation (\ref{PW11}) shows that $ \mathfrak{F}(t) = 0 $ a.e in $ (-\infty, 0 ) _{\mathbb{D}},$ and if $ x_{2}, x_{1} \rightarrow 0,$ then
		\begin{equation}\label{PW12}
			\int_{(0,\infty)_{\mathbb{D}}} \| \mathfrak{F}(t)\|_{k}^{2}  dt \odot dt^{\dagger} \leqslant M.
		\end{equation}
		Thus
		\begin{equation}
			F_{x_{1},x_{2}}(x_{0},x_{3})= \int_{(-\infty,\infty)_{\mathbb{D}}} \hat{F}_{x_{1},x_{2}}(t)\exp\{ it(x_{0},x_{3})\} dt \odot dt^{\dagger}
		\end{equation}
		or
		\begin{align*}
			F(Z) =& \int_{(0,\infty)_{\mathbb{D}}} \mathfrak{F}(t) \exp \{ - ( -ktx_{2} + tx_{1} ) \} \exp\{ it(x_{0} +k x_{3} ) \} dt \odot dt^{\dagger}\nonumber\\
			=& \int_{(0,\infty)_{\mathbb{D}}} \mathfrak{F}(t) \exp\{ itZ\} dt \odot dt^{\dagger} ~~; ~Z \in \substack \prod_{\mathbb{BC}^{+}}.
		\end{align*}
		Keeping $ x_{2}, x_{1}$ fixed and again applying the bicomplex Plancheral Theorem, we obtain
		\begin{align*}
			\frac{1}{2\pi} \int_{(-\infty,\infty)_{\mathbb{D}}} \| F(z_{1} + jz_{2})\|_{k}^{2} dx_{0} =& \int_{(0,\infty)_{\mathbb{D}}} \| \mathfrak{F} \|_{k}^{2} \exp\{ -2(t(-kx_{2} + x_{1})) \} dt \odot dt^{\dagger}\nonumber\\
			\leqslant &  \int_{(0,\infty)_{\mathbb{D}}} \| \mathfrak{F}(t) \|_{k}^{2} dt \odot dt^{\dagger}.\nonumber
		\end{align*}
		Thus
		\begin{equation}\label{PW13}
			\sup_{\substack {{\mathbb{D}}\\ 0 < x_{1}, x_{2} < \infty }} \frac{1} { 2\pi} \int_{(-\infty,\infty)_{\mathbb{D}}} \|F(x_{0} + i x_{1} + j x_{2} + k x_{3} ) \|_{k}^{2} dx_{0} = M \leqslant \int_{(0, \infty)_{\mathbb{D}}} \| \mathfrak{F}(t)\|_{k}^{2} dt \odot dt^{\dagger}.
		\end{equation}
		
		Thus, from equations (\ref{PW12}) and (\ref{PW13}), we get
		$$ \int_{(0,\infty)_{\mathbb{D}}} \| \mathfrak{F} \|_{k}^{2} dt \odot dt^{\dagger} = M.$$
	\end{proof}

	Next, we discuss another class of all bicomplex $ F $ of the form
	\begin{equation}\label{PW14}
		F(Z) = \int_{(-A,A)_{\mathbb{D}}} \mathfrak{F}(t) \exp\{itZ\} dt \odot dt^{\dagger},
	\end{equation}
	where $ \mathfrak{F} \in L_{k}^{2}(-A,A)_{\mathbb{D}} $ and $ A $ are finite and positive. So,
	\begin{align}
		\| F(Z)\|_{k} \leqslant & \sqrt{2} \int_{(-A,A)_{\mathbb{D}}} \|\mathfrak{F}(t)\|_{k}\exp\{ - ( -ktx_{2} + tx_{1} ) \} dt \odot dt^{\dagger} \nonumber\\
		\leqslant & \sqrt{2} \exp\{ A \| ( tx_{1} - ktx_{2} ) \|_{k}\} \int_{(-A,A)_{\mathbb{D}}} \| \mathfrak{F}(t) \|_{k} dt \odot dt^{\dagger}. \label{PW15}
	\end{align}
	If $ C = \sqrt{2} \int_{(-A,A)_{\mathbb{D}}} \| \mathfrak{F}(t) \|_{k} dt ,$ then $ C < \infty $ and  (\ref{PW15}) becomes
	\begin{equation}\label{PW16}
		\| F(Z) \|_{k} \leqslant C \exp\{ A \| Z \|_{k} \}.
	\end{equation}
	We can also prove that $F$ being entire functions that satisfy (\ref{PW16}) are called bicomplex exponential types. The context of our next theorem is as: \\
	The type of functions in equation (\ref{PW14}) are exponential functions whose restrictions to the real and kth-axis lie in $ L_{k}^{2}.$ We prove that the converse is also true.

	\begin{theorem}
		Let $ F$ be a bicomplex function of exponential type and
		\begin{equation}\label{PW17}
			\int_{(-\infty,\infty)_{\mathbb{D}}} \| F ( x_{0} + k x_{3})\|_{k}^{2}dx_{0} \odot dx_{0}^{\dagger} ~<~ \infty.
		\end{equation}
		Then there exists $ \mathfrak{F} \in L_{k}^{2}(-A,A)_{\mathbb{D}} $ such that,
		\begin{equation}\label{PW18}
			F(Z) = \int_{(-A,A)_{\mathbb{D}}} \mathfrak{F}(t) \exp \{ itZ \} dt \odot dt^{\dagger}
		\end{equation}
		for all $ Z \in \mathbb{BC}.$
	\end{theorem}
	
	\begin{proof}
		Let $ \epsilon_{\mathbb{D}} $ be a number greater than 0, and let $ F_{\epsilon_{\mathbb{D}}}( x_{0}+kx_{3}) = F(x_{0}+kx_{3}) \exp\{ - \epsilon_{\mathbb{D}} \| x_{0} + k x_{3}\|_{k} \}.$ Then, we show that
		\begin{equation}\label{PW19}
			\displaystyle \lim_{\mathbb{D}} \\ \epsilon_{\mathbb{D}} \rightarrow 0 \int_{(-\infty,\infty)_{\mathbb{D}}} F_{\epsilon_{\mathbb{D}}} ( x_{0} + kx_{3}) \exp\{ -it(x_{0} + k x_{3} )\} dx_{0} \odot dx_{0}^{\dagger} = 0,
		\end{equation}
		where $ t \in ( -\infty,\infty)_{\mathbb{D}} $ and $ \|t\|_{k}>A.$ As we see that $ \| F_{\epsilon_{\mathbb{D}}} - F \|_{k,2} \rightarrow 0 $ as $ \epsilon_{\mathbb{D}} \rightarrow  0.$ The bicomplex Plancheral Theorem implies that $ \| \hat{F}_{\epsilon_{\mathbb{D}}} - \mathfrak{F} \|_{k,2} \rightarrow 0 $ as $ \epsilon_{\mathbb{D}} \rightarrow 0, $ where $ \mathfrak{F} $ is the bicomplex fourier transform of $ F. $ Thus, equation (\ref{PW19}) implies that $ \mathfrak{F}(t) = 0 $ outside $ [-A,A]_{\mathbb{D}} $ and hence  from Corllary \ref{PW14a} , we see that (\ref{PW18}) holds for almost every $ Z = x_{0} + k x_{3}. $ Also, the left and right-hand sides of the equation (\ref{PW18}) represent the entire bicomplex function. Thus, (\ref{PW18}) holds for every $ Z \in \mathbb{BC} .$ \\
		Thus, in order to prove the theorem, we shall show that (\ref{PW19}) holds. \\
		For this, let $\curlywedge_{\alpha}$ be a bicomplex path, defined as
		$$ \curlywedge_{\alpha}(u) = u ~\exp\{i\alpha\}, $$
		where $ u \in [ 0, \infty)_{\mathbb{D}}.$ Then,
		\begin{equation}\label{PW20}
			\curlywedge_{\alpha}(u) = e \curlywedge_{\alpha_{1}} + e^{\dagger} \curlywedge_{\alpha_{2}},
		\end{equation}
		where $ \curlywedge_{\alpha_1} $ and $ \curlywedge_{\alpha_2} $ are complex paths. Putting, the half-plane in $ \mathbb{BC} $ as
		$$ \substack \prod_{\mathbb{BC}(\alpha)} = \{ W : Re ( W \exp\{ i \alpha \} ) > A \}$$
		and again
		\begin{equation}\label{PW21}
			\substack \prod_{\mathbb{BC}(\alpha)} = e \substack \prod_{\alpha} +e^{\dagger} \substack \prod_{\alpha},
		\end{equation}
		where $ \substack \prod_{\alpha}$ are decomposition of $ \substack \prod_{\mathbb{BC}(\alpha)}$ in complex plane. Define,
		\begin{equation}\label{PW22}
			\Omega_{\alpha}(W) = \int_{\curlywedge_{\alpha}} F(Z) \exp\{ -WZ\} dZ \odot dZ^{\dagger}.
		\end{equation}
		Then $ \Omega_{\alpha}(W) = e \Omega_{\alpha}(W_{1}) + e^{\dagger} \Omega_{\alpha}(W_{2}),$ where 
		$$ \Omega_{\alpha}(W_{i}) = \int_{\curlywedge_{\alpha_{1}}}F_{i}(\beta_{i}) \exp \{ -W_{i}\beta_{i} \} d\beta_{i} \quad  \textnormal{for} \quad  i=1,2.$$
		Using the complex version of this theorem on  \cite[Page 375]{rudin1966}, we see that each $ \Omega_{\alpha}(W_{i}) $ is holomorphic in the half plane $\substack \prod_{\alpha},$ and so $ \Omega_{\alpha}(W) $ is holomorphic in $\substack  \prod_{\mathbb{BC}(\alpha)}.$ Also, if $ \alpha = 0,$ then
		$$ \Omega_{0}(W) = \int_{(0,\infty)_{\mathbb{D}}} F(x_{0}+kx_{3}) exp\{ -W ( x_{0}+kx_{3})\} dx_{0} \odot dx_{0}^{\dagger}~~~~~;\;\;\;ReW > 0 $$
		and if $ \alpha=\pi, $
		$$\Omega_{\pi}(W) = -\int_{(-\infty,0)_{\mathbb{D}}} F(x_{0}+kx_{3}) exp\{ -W ( x_{0}+kx_{3})\} dx_{0} \odot dx_{0}^{\dagger}~~~~~;\;\;\;ReW<0. $$
		Thus, $ \Omega_{0} $ and $ \Omega_{\pi} $ are holomorphic in the indicated half planes in (\ref{PW17}). \\
		Now, if we see that
		\begin{align*}
			\Omega_{0}(\epsilon_{\mathbb{D}} +&it ) - \Omega_{\pi}(-\epsilon_{\mathbb{D}}-it) \\
			& = \int_{(0,\infty)_{\mathbb{D}} } F(x_{0}+kx_{3}) \exp\{ - ( \epsilon_{\mathbb{D}} +it )(x_{0}+kx_{3})\} dx_{0} +\\& \int_{(-\infty,0)_{\mathbb{D}}}F(x_{0}+kx_{3}) \exp\{ - ( - \epsilon_{\mathbb{D}}+ it ) ( x_{0}+ k x_{3}) dx_{0}^{\dagger}\nonumber\\
			=& \int_{(-\infty,\infty)_{\mathbb{D}}} F(x_{0} + k x_{3}) \exp\{ -(\epsilon_{\mathbb{D}} + it ) ( x_{0} + k x_{3} ) - ( - \epsilon_{\mathbb{D}} + it ) ( x_{0}+ k x_{3}) \} dx_{0}^{\dagger} \nonumber\\
			=& \int_{(-\infty,\infty)_{\mathbb{D}}} F(x_{0} + k x_{3}) \exp\{ ( x_{0} + k x_{3} ) [ - \epsilon_{\mathbb{D}} -it + \epsilon_{\mathbb{D}} -it ] \} dx_{0} \odot dx_{0}^{\dagger}\nonumber\\
			=& \int_{(-\infty,\infty)_{\mathbb{D}}} F ( x_{0} + k x_{3} ) \exp \{ ( x_{0} + k x_{3} )(-it) \} dx_{0} \odot dx_{0}^{\dagger}, \nonumber
		\end{align*}
		then it is sufficient to show that $ \Omega_{0} ( \epsilon_{\mathbb{D}}) - \Omega_{\pi} ( - \epsilon_{\mathbb{D}} + it ) \rightarrow  0 $ as $ \epsilon_{\mathbb{D}} \rightarrow 0 $ if $ t > A $ and $ t < - A.$\\
		This can be shown by using the idempotent decomposition of $ \Omega_{0} $and $ \Omega_{\pi} $ with the help of idempotents $ e $ and $ e^{\dagger} $ and also using the fact that this theorem holds for its complex version.\\
		Therefore,
		$$\displaystyle \lim_{\mathbb{D}} \\ \epsilon_{\mathbb{D}} \rightarrow 0 \int_{(-\infty,\infty)_{\mathbb{D}}} F_{\epsilon_{\mathbb{D}}} ( x_{0} + k x_{3} ) \exp\{ -it(x_{0}+kx_{3})\} dx_{0} = 0. $$
	\end{proof}
	Now, we prove the bicomplex Cauchy integral formula for the upper half-plane. We start with the following statement:
	\begin{theorem}
		
		If $ F \in H^{p}(\substack \prod_{\mathbb{BC}^{+}})$ ; $ 1 \leqslant p < \infty,$ then
		$$F(Z) = \frac{1}{2\pi i} \int_{(-\infty,\infty)_{\mathbb{D}}} \frac{ F(W) }{ W - Z } dW \odot dW^{\dagger}~~; Z ~ \in~ \substack \prod_{\mathbb{BC}^{+}}$$
		and the integral vanishes for all $ Z \in \substack \prod_{\mathbb{BC^{-}}}.$ \\
		Conversely, if $ H \in L_{k}^{q}(dm) ~( 1 \leqslant p < \infty ) $ and
		$$ \frac{1}{2\pi i } \int_{(-\infty,\infty)_{\mathbb{D}}} \frac{H(W)}{W-Z} dW \odot dW^{\dagger} = 0 $$
		for all $ Z \in \substack \prod_{\mathbb{BC^{-}}}.$ Then for $ Z \in \substack \prod_{\mathbb{BC^{+}}},$ this integral represents a bicomplex function $ F \in H^{p}(\substack \prod_{\mathbb{BC^{+}}}), $ where the boundary function
		$$ F(x_{0}, x_{3}) = H(x_{0},x_{3})~~a.e.$$
	\end{theorem}
	\begin{proof}
		
		Since $ H^{p}(\substack \prod_{\mathbb{BC^{+}}}) = e H^{P}(\substack \prod^{+})+e^{\dagger} H^{p}(\substack \prod^{+}) $ and $ F \in H^{p}(\substack \prod_{\mathbb{BC^{+}}}).$ Then the bicomplex Cauchy integral Formula, see \cite{reyes2019}, is given by
		$$ C(F(Z)) = e \frac{1}{2\pi i} \int_{-\infty}^{\infty} \frac{F_{1}(W_{1})}{W_{1}-\beta_{1}} dW_{1} + e^{\dagger} \frac{1}{ 2 \pi i} \int_{-\infty}^{\infty} \frac{F_{2}(W_{2})}{W_{2}-\beta_{2}} dW_{2}.$$
		That is, $$  C (F(Z)) = e I_{1} + e^{\dagger} I_{2},$$
		where $ I_{1} $ and $ I_{2}$ are the complex Cauchy integrals for $ F_{1},F_{2} \in H^{p}(\substack \prod^{+})$ which is analytic in both $ \substack \prod^{+}$ and $ \substack \prod^{-}.$ Then the bicomplex Cauchy integral is holomorphic in both $\substack  \prod_{\mathbb{BC^{+}}}$ and $\substack \prod_{\mathbb{BC^{-}}}.$ \\
		Now, using the idempotent decompositions, we have
		\begin{align*}
			C(F(Z)) - C(F(Z^{*})) =& \{ e \mathfrak{F}_{1}(\beta_{1}) + e^{\dagger} \mathfrak{F}_{2}(\beta_{2}) \} - \{ e \mathfrak{F}_{1}(\bar{\beta_{1}}) + e^{\dagger} \mathfrak{F}_{2}(\bar{\beta_{2}})\}\nonumber\\
			=&  e \{ \mathfrak{F}_{1}(\beta_{1}) - \mathfrak{F}_{1}(\bar{\beta_{1}})\} + e^{\dagger} \{\mathfrak{F}_{2} (\beta_{2}) - \mathfrak{F}_{2} (\bar{\beta_{2}})\}.\nonumber
		\end{align*}
		So, from the complex analogy of this theorem, we have,
		\begin{align*}
			C(F(Z)) - C(F(Z^{*})) =&  e \{ F_{1}(\beta_{1}) \} + e^{\dagger} \{ F_{2}(\beta_{2})\}~~\beta_{1},\beta_{2}\in \substack \prod^{+}\nonumber\\
			=& F(Z)~;~Z\in \substack \prod_{\mathbb{BC^{+}}}.\nonumber
		\end{align*}
		Thus, $ C(F(Z^{*}) $ is holomorphic for $ Z \in \substack  \prod_{\mathbb{BC^{+}}}.$ So, $ C(F(Z)) $ must be identically constant in $ \substack  \prod_{\mathbb{BC^{-}}}.$ Since $ C (F(Z)) \longrightarrow 0 $ as $ Z \longrightarrow \infty, $ we have
		$$ C(F(Z)) = 0. $$
		Thus,
		$$ C(F(Z^{*})) = F(Z) \in \substack \prod_{\mathbb{BC^{+}}}$$
		and $$ C(F(Z)) = 0  \in \substack\prod_{\mathbb{BC^{-}}}.$$
		Conversely, suppose $ H \in L_{k}^{q}(m).$ Then
		\begin{equation}\label{CT1}
			L_{k}^{q}(m) = e L^{q}(m_{1}) + e^{\dagger} L^{q}(m_{2})
		\end{equation}
		and
		$$ \frac{1}{2 \pi i } \int_{(-\infty, \infty)_{\mathbb{D}}} \frac{H(W)}{W-Z} dW \odot dW^{\dagger} = 0 ~~; Z \in \substack \prod_{\mathbb{BC^{+}}}.$$
		Since $ H \in L_{k}^{q}(m) $ and using the decomposition in equation  (\ref{CT1})  and using the fact that the result holds in each $ L^{q}(m_{i});i=1,2,$  the theorem follows.
		
	\end{proof}

\end{document}